\documentclass[11pt,a4paper]{article}
\usepackage[T1]{fontenc}
\usepackage[utf8]{inputenc}
\usepackage[margin=1in]{geometry}
\usepackage{tabularx, enumerate,enumitem, authblk,soul}
\usepackage[dvipsnames]{xcolor}
\usepackage{amsmath,amssymb,dsfont,subcaption}
\usepackage{mathtools, graphicx,color,dsfont,amsthm,bbm, mathrsfs}
\usepackage{lipsum,bm,comment}
\usepackage{natbib}
\usepackage{longtable}
\usepackage{array}
\usepackage{caption}
\usepackage[colorlinks=true,linkcolor=blue,anchorcolor=blue,citecolor=blue,filecolor=black,menucolor=black,runcolor=black,urlcolor=black]{hyperref}
\usepackage{cleveref}

\theoremstyle{plain}
\newtheorem{Theor}{Theorem}[section]
\newtheorem{Lem}[Theor]{Lemma}

\newtheorem{Prop}[Theor]{Proposition}
\newtheorem{Corol}[Theor]{Corollary}

\theoremstyle{definition}

\newtheorem{Remark}[Theor]{Remark}

\newcommand{\E}{\mathbb{E}}

\newcommand{\N}{\mathbb{N}}

\newcommand{\R}{\mathbb{R}}

\renewcommand{\AA}{\mathcal{A}}

\newcommand{\DD}{\mathcal{D}}

\newcommand{\KK}{\mathcal{K}}

\newcommand{\MM}{\mathcal{M}}

\newcommand{\NNN}{\mathscr{N}}

\def\1{{\mathbb I}}

\newcommand{\ie}{\textit{i.e.} }

\newcommand{\ps}[2]{\left\langle #1,#2 \right\rangle}

\newcommand{\ve}{\varepsilon}
\newcommand{\Om}{\Omega}

\newcommand{\spa}{\quad\quad}

\newcommand{\embed}{\hookrightarrow}

\usepackage{scalerel,stackengine}
\stackMath
\newcommand\wwidehat[1]{%
	\savestack{\tmpbox}{\stretchto{%
			\scaleto{%
				\scalerel*[\widthof{\ensuremath{#1}}]{\kern-.6pt\bigwedge\kern-.6pt}%
				{\rule[-\textheight/2]{1ex}{\textheight}}
			}{\textheight}%
		}{0.5ex}}%
	\stackon[1pt]{#1}{\tmpbox}%
}

\newcommand{\RN}[1]{%
	(\textup{\uppercase\expandafter{\romannumeral#1}})%
}
\newcolumntype{H}{>{\setbox0=\hbox\bgroup}c<{\egroup}@{}}

\newenvironment{Proof}[1]
{
	\begin{proof}[Proof of #1]
	}
	{
	\end{proof}
}

\allowdisplaybreaks

\title{On statistical inference for non-linear dynamical systems evolving in their global attractor}

\author{Dimitri Konen and Richard Nickl}
\affil{University of Cambridge}

\date{\today}

\begin{document}
	
\maketitle

\begin{abstract}
We consider a periodic scalar  reaction-diffusion system over a two-dimensional domain under natural conditions on the reaction function and with initial condition $\theta$. We show that on the global attractor $\mathcal A$ of the resulting dynamical system $(u_\theta(t):t>0)$, a reverse Poincar\'e inequality holds true, and that as a consequence the map $\theta \mapsto u_\theta(t)$ satisfies an $L^2(\Om)$-Lipschitz stability estimate on $\mathcal A$ for any $t>0$ fixed. We then show that statistical recovery of an initial condition $\theta$ in the attractor $\mathcal A$, as well as prediction of the states $u_\theta$, is possible from discrete measurements of the system at parametric $1/\sqrt N$ convergence rates.
\end{abstract}

\setcounter{tocdepth}{2}
{
\hypersetup{linkcolor=black}
}

\section{Introduction}


Consider a non-linear dynamical system $(u(t, \cdot): t >0)$ evolving in the Hilbert space $L^2(\Omega)$ of square integrable functions over some bounded domain $\Omega$ in Euclidean space. The focus here will be the situation when the infinitesimal dynamics are described by a non-linear parabolic partial differential equation (PDE) of the form
\begin{equation}\label{pde0}
\begin{aligned}
    \frac{\partial u}{\partial t} - \Delta u - F(u) &=0,  \\
    u(0,\cdot)&=\theta, 
\end{aligned}
\end{equation}
where $\Delta$ is the Laplacian (acting on $x \in \Omega$), where the functional $F$ is given and models the non-linearity, and where $\theta \in L^2(\Omega)$ is an unknown initial condition. We consider the scenario where the system is in its equilibrium dynamics, specifically, when it evolves in its \textit{global attractor}, which consists of all the states of the system that can be reached as $t \to \infty$. If we denote by $\{S_t:t\ge 0\}$ the semigroup  $S_t\theta=u_\theta(t)$ associated with the above PDE, then under certain assumptions on $F$ the long-time dynamics `collapse' into a compact subset of $L^2(\Omega)$. Specifically, if $B$ is a compact set in $L^2(\Omega)$ such that $S_tX\subset B$ for all bounded $X\subset L^2(\Omega)$ and all $t\ge t_0(X)$ large enough, then the set 
\begin{equation}\label{lazydimitri}
\AA 
\equiv 
\bigcap_{t> 0} S_t B 
\end{equation}
is nonempty, compact and invariant in the sense that $S_t\AA=\AA$ for all $t\ge 0$. Thus once the dynamics reach $\mathcal A$, they never leave this region again, and $\mathcal A$ describes a family of long-term invariant states known as the `global attractor' of the dynamical system -- they can be seen to be independent of the choice of $B$. It can be shown to exist in a variety of non-linear `dissipative' PDEs including reaction-diffusion and Navier-Stokes equations, and generally has a finite but positive Hausdorff dimension, see \cite{BV92, T97, Robinson2001}.

\smallskip

In this article we are concerned with the problem of how to estimate or predict the states $(u_\theta(\tau):\tau \ge 0)$ of the system from discrete statistical observations, thus contributing to a recently emerging literature on nonlinear statistical inference problems with parabolic PDEs, see \cite{N24, NicTit2024, Nickl2024, NPR25, KSV24, ABJN25, GW25, Kon25Minimax, NS26, MvdMvdV26, CN26} and references therein. We consider the case where the system has reached its `equilibrium dynamics', that is, when $u_\theta$ evolves in $\mathcal A$.  The measurement setting we consider is the following: For $t>0$ fixed, data $Z^{(N)} = (Y_i, X_i)_{i=1}^N$ is obtained through the regression model
\begin{equation}\label{model}
Y_i 
=
u_{\theta}(t, X_i)+ \ve_i
,
\end{equation}
where the initial condition $\theta$ is known to lie in $\mathcal A$, where the $X_i$'s are independent and identically distributed (i.i.d.)~drawn from the uniform distribution over $\Om$, and the $\ve_i$'s are i.i.d.~standard normal $\NNN(0,1)$ random variables accounting for measurement errors, independent of the $X_i$'s. Let $P_\theta$ denote the law in $\R\times \Om$ of $(Y_1, X_1)$ with corresponding infinite product measure $P_\theta^{\N}$ whose $N$-dimensional marginal distributions describe the law of $Z^{(N)}$.  

\smallskip

It is a common practice in data assimilation models such as (\ref{model}) to assign a Gaussian random field to the initial condition $\theta$ and to then use the push-forward under the dynamics as a `prior' probability measure in the space of trajectories of $(u_\theta(t, \cdot): t>0)$. A typical situation \cite{S10, CDRS09, RC15, LSZ15} is to maintain an infinite-dimensional model for the initial condition $\theta$ which can be an arbitrary element of a Sobolev space $\theta \in H^\alpha$ for suitable $\alpha>0$. This explicitly declares the initial condition $\theta$ to be the parameter modelling the uncertainty about the dynamical system, and in proving theorems about the statistical performance of posterior measures arising from such priors, the inverse problem of identifying the initial condition $\theta$ from the measurements of $u_\theta$ then takes center stage. Recovery of the initial condition is generally a severely ill-posed problem in the measurement model (\ref{model}) with $t>0$ fixed, and the best available minimax statistical convergence rates for $\theta$ belonging to a fixed Sobolev ball are `slow', that is, of order $(1/\log N)^a, a>0,$ where $N$ is the sample size. This was proved in Theorems~2 and 4 in \cite{NicTit2024} for the periodic 2D Navier-Stokes equations but holds also more generally (at least in perturbative neighbourhoods of the standard heat equation, i.e., when $F$ in (\ref{pde0}) is small in an appropriate sense). When measurements for times near zero are available, faster (algebraic in $1/N$) convergence rates are available as was proved in the recent articles \cite{Nickl2024, KonNic26}. In the present work we identify another situation where slow logarithmic rates can be improved upon, namely when the initial condition $\theta$ is known to belong to the global attractor $\mathcal A$. We will show that in this case statistical estimators (i.e., measurable functions of the data $Z^{(N)}$) exist that themselves belong to $\mathcal A$ and that infer the initial condition and then also $u_\theta(\tau), \tau>0,$ at parametric convergence rates $1/\sqrt N$, even when no samples near $t=0$ are available. 

\smallskip

The mathematical proofs of these facts will be executed in a specific periodic reaction-diffusion model over a two-dimensional domain, where we show that a key reverse Poincar\'e inequality, which gives a Lipschitz stability estimate for the inverse problem similar in spirit to Theorem 1B in \cite{NicTit2024}, can be established. To introduce the PDE, for $\theta\in L^2(\Om)$, consider the unique solution $u=u_\theta:[0,\infty)\times \Om\to \R$ to the reaction-diffusion equation
\begin{equation}\label{eq:ReacDiff}
\begin{aligned}
\frac{\partial u}{\partial t}(t,x) -\Delta u(t,x) &= f(u(t,x)) 
\quad \text{for } (t,x) \in (0,\infty)\times \Omega, \\ 
u(0,x) &= \theta(x) 
\quad\quad~~~~ \text{for } x \in \Om
\end{aligned}
\end{equation}
where $\Omega = (0,1]^2$ is the two-dimensional torus with unit area and opposite edges identified. The reaction term  $f:\R\to\R$ is of class $C^3(\R)$ -- throughout, $C^k(\R)$ denotes the class of $k$-times continuously differentiable maps defined on $\R$, without any boundedness assumption -- and further satisfies, for some $p>2$, and constants $\alpha_1 > \alpha_2 > 0$, $k\ge 0$ and $L\in \R$, the inequalities
\begin{equation}\label{Condf}
-k - \alpha_1 |s|^p
\le 
f(s)s
\le 
k - \alpha_2 |s|^p,
\quad 
\textrm{and} 
\quad
f'(s)\le L 
,\spa 
\forall\ s\in \R.
\end{equation}
Under these assumptions, the system of equations in (\ref{eq:ReacDiff}) possesses a unique (weak) solution satisfying $u_{\theta} \in C([0,\infty), L^2(\Om))$; we review this in Proposition \ref{PropSolExist} below. For each $t\ge 0$ define the solution operator $S_t : L^2(\Om)\to L^2(\Om),\ \theta \mapsto u_{\theta}(t,\cdot)$. Then $\{S_t : t\ge 0\}$ is a semigroup on $L^2(\Om)$, and the semidynamical system $(L^2(\Om), \{S_t\}_{t\ge 0})$ admits a global attractor~$\AA$ as we explain in Proposition \ref{PropAttractor} below. We discuss in Remark \ref{dimitridoesallthework} concrete examples for $f$ such that the corresponding attractor $\mathcal A$ has strictly positive and finite Hausdorff dimension.

\smallskip

Our main statistical result is concerned with convergence rates as sample size $N \to \infty$ for the problem of inferring an initial condition $\theta \in \AA$ and with predicting $u_\theta(\tau)$ evolving in $\AA$ at arbitrary times $\tau>0$.

\begin{Theor}\label{mainintro}
    Assume that $\theta_0\in \AA$ and consider i.i.d.~data $(Y_i, X_i)_{i=1}^N$ obtained through (\ref{model}) for some $t>0$ where $u_\theta$ solves the PDE (\ref{eq:ReacDiff}) with known $f \in C^3(\R)$ satisfying (\ref{Condf}). Then, one can construct estimators $\hat \theta_N = \hat \theta(Y_i, X_i)_{i=1}^N$ such that for every $\tau\ge 0$ we  have as $N \to \infty$
   $$ \|\hat \theta_N - \theta_0\|_{L^2(\Omega)} = O_{P_{\theta_0}^\N}\Big( \frac{e^{ct}}{\sqrt N}\Big),~~\text{and in fact}~
    \|u_{\hat \theta_N}(\tau) - u_{\theta_0}(\tau)\|_{L^2(\Omega)} =O_{P_{\theta_0}^\N}\Big( \frac{e^{c|\tau-t|}}{\sqrt N}\Big),
    $$
    for some fixed $c>0$ depending only on $f$.
\end{Theor}

The preceding theorem follows from Theorem \ref{TheorContractTheta}, to be proved below. It shows that $\hat \theta_N$ and $u_{\hat \theta_N}$ themselves take values in $\mathcal A$. The proof further shows that the above convergence rate is attained uniformly over $\theta_0 \in \mathcal A$.

Since the attractor $\mathcal A$ has the metric complexity of a bounded subset of a finite-dimensional space, the parametric convergence rate $1/\sqrt N$ obtained in the preceding theorem is not unnatural -- similar results were obtained in other severely ill-posed non-linear inverse problems such as the Calder\'on problem, see \cite{AN19, B23}. However, unlike in these references, even though $\mathcal A$ is in a certain sense finite-dimensional, it is not clear that local asymptotically normal (LAN) expansions of $P_{\theta+h}$ in directions $h$, and the related notion of Fisher information, make sense on $\mathcal A$ due to its lack of linear structure.  Hence techniques from classical finite-dimensional statistics (\cite{van1998}) do not apply straightforwardly in the present context. Instead the proof of Theorem \ref{mainintro} follows the strategy employed in non-linear regression models in \cite{MNP21a},  \cite{NicklEMS} to first construct a sufficiently good estimator that lies in the range of the forward map and to then combine this with a stability estimate for the non-linear regression parameterisation $\theta \mapsto u_\theta$.  This result is given in the following theorem: It shows that (with the exception of the dependence of the constant on $t$) the ill-posed nature of that inverse problem disappears when the dynamics has equilibrated to the global attractor $\mathcal A$.

\begin{Theor}\label{backstab}
   Let $u_\theta$ be as in Theorem \ref{mainintro}. For all $t>0$, there exists a positive constant $c=c(t,f)$ depending only on $t, f$ such that 
    $$
    \|\theta-\vartheta\|_{L^2(\Om)} 
    \le 
    c \|u_{\theta}(t)-u_{\vartheta}(t)\|_{L^2(\Om)}
    ,\spa 
    \forall\ \theta,\vartheta\in \AA 
    .
$$
\end{Theor}

The dependence of the constant $c$ on $t$ is given in Proposition~\ref{PropStability} and Corollary \ref{TheorStability}. The proof is inspired by Theorem 1B in \cite{NicTit2024} for the Navier-Stokes equations where one assumes a \textit{reverse Poincar\'e inequality} $\|\nabla(\theta-\vartheta)\|_{L^2(\Om)} \lesssim \|\theta-\vartheta\|_{L^2(\Om)}$. For $\theta, \vartheta$ belonging to the global attractor $\AA$ of the PDE~(\ref{eq:ReacDiff}), this inequality  holds true as we show in Proposition \ref{PropReversePoin} below, using ideas from \cite{Foias1989}. The proof relies on certain properties of the distribution of the eigenvalues of the Laplacian and of the non-linearity $F$ featuring in (\ref{pde0}) which are specific to the periodic setting over a  two-dimensional domain. Proving similar results for other dimensions or the Navier-Stokes equations remains a formidable challenge, but we conjecture that similar theorems do hold true in these settings as well.

\begin{Remark}[Discussion of Condition (\ref{Condf})] \label{dimitridoesallthework} 
Because $f$ is assumed to be continuous over $\R$ it is locally bounded, and then a straightforward calculation shows that the two-sided inequality in~(\ref{Condf}) equivalently rewrites in the following more interpretable form: 
$$
f(s) 
=
-\alpha s|s|^{p-2}
+
g(s)
,\spa 
\text{with}~~ 
|g(s)|
\le
C + \delta |s|^{p-1}
,
$$
for some $\alpha>\delta>0$, and some $C>0$. Examples of such reaction functions $f$ include $f(s)=\sum_{i=1}^{p-2} h_i(s) s^i + c_{p-1} s |s|^{p-2}$ with $p \ge 4$ even (to ensure smoothness of $f$) and negative $c_{p-1}<0$, for smooth bounded and Lipschitz functions $h_i$'s. 
For such functions, the condition $f'\le L$ can be enforced by requiring further boundedness conditions on $h_{p-2}'$ or simply $h_{p-2}'\equiv 0$. \smallskip

Let us remark, to avoid triviality, that for such $f$ the global attractor $\mathcal A$ exists (see \cite{Marion1987}, or also Proposition~\ref{PropAttractor} below) and can further be seen to have non-zero dimension: This is granted for instance for $f$ satisfying (\ref{Condf}), as soon as $f'(0)$ is sufficiently large and $f(0)=0$. Indeed, Theorem~9.4 in \cite{BV83} imply that $\AA$ has positive Hausdorff dimension as soon as $f$ is of the form $f(s)=\lambda s + h(s)$ with sufficiently large~$\lambda$ and smooth $h$ satisfying (\ref{Condf}) in place of $f$ and such that $h(0)=h'(0)=0$. This can be seen by applying these results with $f$ there given by $-h(s)$, $n=2$, $\lambda$ large enough, and $g=g_1=0$, also taking note that $h(u)\in C^\ell(\Om)$ for any $\ell\in\N$ because $h$ is smooth hence also $u$ by virtue of classical regularity theory for reaction-diffusion equations (see Theorem~11.8 in \cite{Robinson2001}). In particular, $\AA$ has non-trivial dimension for $f(s)=\sum_{i=1}^{p-1} c_i s^i$ with $p\ge 4$ even, $c_{p-1}<0$ and large enough $c_1$, and more generally for $f(s) = \sum_{i=1}^{p-2} h_i(s) s^i + c_{p-1} s^{p-1}$ with $p\ge 4$ even, $h_i$ as before with $h_{p-2}'\equiv 0$, $c_{p-1}<0$, and $h_1(0)=f'(0)$ large enough. Of course for specific choices of $f$ (such as in the Chafee-Infante equation), explicit constructions of non-empty $\mathcal A$ can often be given directly. \smallskip

The intuition behind hypothesis (\ref{Condf}) is, on the one hand, that the forcing $f(u)$ behaves like $-|u|^{p-2}u$ hence counterbalances $u$ when $|u|$ is large, thus `taking away energy in the system' to prevent blowups. On the other hand, when $|u|$ is close to zero and $f(0)=0$, the forcing $f(u)$ behaves like $f'(0)u$, preventing the dynamics from dying out by adding energy proportional to $f'(0)$.

\end{Remark}

\section{Proofs of main results}

Theorem~\ref{mainintro} follows from the Lipschitz estimate in Proposition~\ref{PropSolExist} combined with the first bound in Theorem~\ref{TheorContractTheta} and Markov's inequality, whereas Theorem~\ref{backstab} is established in Corollary~\ref{TheorStability}.\medskip 

Our first result is to show that the stability estimate in Theorem~\ref{backstab} holds over a subset $\KK\subset H^1(\Om)$ for the reaction-diffusion system (\ref{eq:ReacDiff}) when a key reverse Poincaré inequality can be verified over $\KK$. We will establish below that such an inequality holds when $\KK=\AA$ is the global attractor or when $\KK=\MM$ is an inertial manifold constructed in Section~\ref{sec:Manifold} below. Other classes of functions satisfying a reverse Poincaré inequality include band-limited functions for the spectrum of the Laplace operator, similar to Remark 1 in \cite{NicTit2024}.


\begin{Prop}\label{PropStability}
Let $\KK$ be a subset of $H^1(\Om)$ such that
$$
\|\nabla(\theta-\vartheta)\|_{L^2(\Om)}\le C_\KK \|\theta-\vartheta\|_{L^2(\Om)}
,\spa 
\forall\ \theta,\vartheta\in \KK 
,
$$
for some constant $C_\KK>0$. If, in addition, we have for some $R>0$
$$
\sup_{\theta\in\KK} \sup_{t\ge 0} \|u_\theta(t)\|_{L^\infty(\Om)}
\le 
R
,
$$
then, letting $c_{f,R} \equiv \sup_{s\in [-R,R]} |f'(s)|$, we have
$$
\|\theta-\vartheta\|_{L^2(\Om)} 
\le 
\exp\Big(\frac14 c_{f,R}^2 t^2 + (c_{f,R} + C_\KK^2)t\Big) \|u_{\theta}(t)-u_{\vartheta}(t)\|_{L^2(\Om)}
,\spa 
\forall\ t\ge 0,\ 
\forall\ \theta,\vartheta\in \KK
.
$$
\end{Prop}

This result holds for general subsets $\KK$ of $H^1(\Om)$ including, in particular, the global attractor~$\AA$ and the inertial manifold $\MM$ constructed in Section \ref{sec:Manifold}, as we show in Corollary~\ref{TheorStability} below. The proof of Proposition \ref{PropStability} proceeds by controlling the growth of the so-called `Dirichlet ratio' 
$$
\phi(t)
\equiv 
\frac{\|\nabla (u_\theta(t)-u_\vartheta(t))\|^2_{L^2(\Om)}}{ \|u_\theta(t)-u_\vartheta(t)\|^2_{L^2(\Om)}}
$$
in an analogous way as the proof of Theorem 1B in \cite{NicTit2024}. When $\theta,\vartheta\in\AA$ already initially belong to the global attractor, however, the dynamics $u_\theta(t)$, $u_\vartheta(t)$ stay in $\AA$ at all times so that the global reverse Poincaré inequality on $\AA$ established in Proposition~\ref{PropReversePoin} below, with constant $C_\AA$, immediately provides $\phi(t)\le C_\AA^2$. In this case, the constant in the stability estimate above reduces to $\exp((c_{f,R}+C_\AA^2) t)$.

\begin{Proof}{Proposition \ref{PropStability}}
Let $w(t)\equiv u_\theta(t)-u_\vartheta(t)$ for all $t\ge 0$,
and observe that $w$ satisfies the time-evolution equation
\begin{equation}\label{pseudlin}
\frac{dw}{dt} - \Delta w 
= 
f(u_\theta(t)) - f(u_\vartheta(t))
\equiv 
g(t)
,\spa 
w(0) = u_\theta(0) - u_\vartheta(0) 
.
\end{equation}
Since $\theta,\vartheta\in H^1(\Om)$, we have $w\in L^2([0,T], H^2(\Om))$ by virtue of Proposition~\ref{PropSolExist}. We then have $w'\in L^2([0,T], L^2(\Om))$ since $w'=\Delta w + g$ with $g\in L^2(\Om)$ by (\ref{eq:Boundg}) below. In particular, the solution $w$ is strong and (\ref{pseudlin}) holds in $L^2(\Om)$ for almost all $t>0$. Taking the $L^2(\Om)$-inner product of (\ref{pseudlin}) with $w$ yields
\begin{equation}\label{eq:timeDer}
\frac{1}{2} \frac{d}{dt}\|w(t)\|_{L^2(\Om)}^2 +  \|\nabla w(t)\|_{L^2(\Om)}^2 = \ps{g(t)}{w(t)}_{L^2(\Om)}
.
\end{equation}
We can assume $\theta-\vartheta=w(0) \ne 0$ in $L^2(\Om)$ and by backwards uniqueness (arguing, e.g., as in Theorem~11.10 of \cite{Robinson2001}) also $w(t) \ne 0$ in $L^2(\Om)$ for all $t\ge 0$. Therefore, dividing the previous display by $\|w(t)\|^2_{L^2(\Om)}$ provides 
\begin{equation}\label{eq:LogEnegery}
\frac12 \frac{d}{dt} \log\big( \|w(t)\|^2_{L^2(\Om)}\big) 
=
\frac{\ps{g(t)}{w(t)}_{L^2(\Om)}}{\|w(t)\|^2_{L^2(\Om)}} - \frac{\|\nabla w(t)\|^2_{L^2(\Om)}}{\|w(t)\|^2_{L^2(\Om)}}
.
\end{equation}
By assumption, we have 
$$
\|u_\theta(t)\|_{L^\infty(\Om)}
\le 
R,\spa 
\text{and}
\spa 
\|u_\vartheta(t)\|_{L^\infty(\Om)}\le R 
,\spa 
\forall\ t\ge 0 
.
$$
Since $f$ is continuously differentiable over $\R$, it is Lipschitz on the interval $[-R, R]$ to the effect that 
\begin{equation}\label{eq:Boundg}
|g(t)| 
\le
\|f'\|_{L^\infty([-R, R])} |u_\theta(t) - u_\vartheta(t)| 
\equiv 
c_{f,R} |w(t)|
.
\end{equation}
It follows that
\begin{equation}\label{eq:ineqLog}
\frac12 \frac{d}{dt} \log\big( \|w(t)\|^2_{L^2(\Om)}\big) 
\ge 
- c_{f,R} - \frac{\|\nabla w(t)\|^2_{L^2(\Om)}}{\|w(t)\|^2_{L^2(\Om)}}
.
\end{equation}
We now follow the approach from Theorem 1B in \cite{NicTit2024} which consists in controlling the growth over time of the Dirichlet ratio
$$
\phi(t) 
\equiv 
\frac{\|\nabla w(t)\|^2_{L^2(\Om)}}{\|w(t)\|^2_{L^2(\Om)}}
.
$$
Observe that the time-derivative $(\|w(t)\|^2_{L^2(\Om)})'$ is given in (\ref{eq:timeDer}) and that
\begin{eqnarray*}
\lefteqn{
(\|\nabla w(t)\|^2_{L^2(\Om)})'
=
(\ps{-\Delta w(t)}{w(t)}_{L^2(\Om)})' 
}
\\[2mm]
&&
= 2 \ps{w'(t)}{-\Delta w(t)}_{L^2(\Om)} = 2\ps{\Delta w(t) + g(t)}{-\Delta w(t)}_{L^2(\Om)}
.
\end{eqnarray*}
Note that $\ps{w'}{-\Delta w}_{L^2(\Om)}$ and $\ps{\Delta w}{-\Delta w}_{L^2(\Om)}$ are well defined because $w\in L^2([0,T], H^2(\Om))$ and $w'\in L^2([0,T], L^2(\Om))$. Then, a straightforward computation gives 
$$
\phi'(t) 
=
\frac{2}{\|w\|^4_{L^2(\Om)}} 
\Big\{ 
\ps{\Delta w + g}{-\Delta w}_{L^2(\Om)}\|w\|^2_{L^2(\Om)} - \ps{\Delta w +g}{w}_{L^2(\Om)} \ps{-\Delta w}{w}_{L^2(\Om)}
\Big\} 
.
$$
Because $\ps{-\Delta w}{w}_{L^2(\Om)} = \|\nabla w\|^2_{L^2(\Om)}$, the previous display rewrites 
$$
\phi'(t) 
=
\frac{2}{\|w\|^2_{L^2(\Om)}} 
\Big\{ 
\ps{\Delta w + g}{-\Delta w}_{L^2(\Om)} - \ps{\Delta w +g}{w}_{L^2(\Om)} \phi(t)
\Big\} 
.
$$
Since $\phi$ does not depend on the space variable, we can rearrange the bracketed terms as 
$$
\ps{\Delta w + g}{-\Delta w - w\phi}_{L^2(\Om)}
=
\ps{\Delta w}{-\Delta w - w\phi}_{L^2(\Om)} + \ps{g}{-\Delta w - w\phi}_{L^2(\Om)}
.
$$
Now observe that 
\begin{align*}
\ps{w\phi}{-\Delta w - w\phi}_{L^2(\Om)}
&=
\phi\times \Big( \ps{w}{-\Delta w}_{L^2(\Om)} - \ps{w}{w}_{L^2(\Om)}\phi\Big) 
\\[2mm]
&=
\phi\times \Big( \|\nabla w\|^2_{L^2(\Om)} - \|w\|^2_{L^2(\Om)} \frac{\|\nabla w\|^2_{L^2(\Om)}}{\|w\|^2_{L^2(\Om)}} \Big)
=
0
.
\end{align*}
In particular, we have 
\begin{align*}
\ps{\Delta w + g}{-\Delta w - w\phi}_{L^2(\Om)}
&=
\ps{\Delta w + w\phi}{-\Delta w - w\phi}_{L^2(\Om)} + \ps{g}{-\Delta w - w\phi}_{L^2(\Om)}
\\[2mm]
&= 
-\|\Delta w + w\phi\|^2_{L^2(\Om)}
+
\ps{g}{-\Delta w - w\phi}_{L^2(\Om)}
.
\end{align*}
It follows that 
\begin{align}\label{eq:PhiPrime}
\phi'(t) 
&\le 
\frac{2}{\|w\|^2_{L^2(\Om)}} \Big\{ -\|\Delta w + w\phi\|^2_{L^2(\Om)}
+
\|g\|_{L^2(\Om)}\|\Delta w + w\phi\|_{L^2(\Om)}\Big\}
\\[2mm]
&\le
\frac{2}{\|w\|^2_{L^2(\Om)}} \times \frac{\|g\|^2_{L^2(\Om)}}{4}
.
\end{align}
We already observed that $|g(t)|\le c_{f,R} |w(t)|$, so that we can deduce that 
$$
\phi'(t) 
\le 
\frac{c_{f,R}^2}{2}
,\spa 
\forall\ t\ge 0
.
$$
It follows that $\phi(t)\le \phi(0) + c_{f,R}^2 t/2$ which, since $\phi(0)\le C^2_\KK$ by assumption, provides 
$$
\phi(t) \le C_\KK^2 + \frac{c_{f,R}^2 t}{2},\spa 
\forall\ t\ge 0 
.
$$
Plugging this in (\ref{eq:ineqLog}) yields 
$$
\frac12 \frac{d}{dt} \log\big( \|w(t)\|^2_{L^2(\Om)}\big) 
\ge 
- c_{f,R} - C_\KK^2 - \frac{c_{f,R}^2 t}{2}
.
$$
Integrating over $[0,t]$ thus gives
$$
\log\bigg(\frac{\|w(t)\|^2_{L^2(\Om)}}{\|w(0)\|^2_{L^2(\Om}}\bigg)
\ge 
-2(c_{f,R}+C_\KK^2)t - \frac{c_{f,R}^2 t^2}{2}
,
$$
which provides the conclusion.
\end{Proof}

The proof of Proposition \ref{PropStability} is not specific to the particular reaction-diffusion model and can be used to obtain stability estimates similar to Proposition~\ref{PropStability} in general time evolution equations
\begin{equation}\label{eq:ParabolicPde}
\frac{du}{dt} + A(u) = N(u)
,
\end{equation}
for subsets $\KK\subset H^1(\Om)$ satisfying the reverse Poincaré inequality $\|A^{1/2}(u-v)\|_{L^2(\Om)}\le C_\KK \|u-v\|_{L^2(\Om)}$, where $A$ is a linear positive self-adjoint differential operator with compact inverse, and $N$ is a non-linear term that is of lower order than $A$; this also encompasses in particular the Navier-Stokes system considered in \cite{NicTit2024} and \cite{KonNic26}. \smallskip



We now establish that the key reverse Poincaré inequality required in Proposition~\ref{PropStability} holds when $\KK$ is the global attractor $\AA$ or the inertial manifold $\MM$ constructed in Section~\ref{sec:Manifold} below.

\begin{Prop} \label{PropReversePoin}
There exists a constant $C_\AA>0$ such that 
$$
\|\nabla(\theta-\vartheta)\|_{L^2(\Om)}
\le
C_\AA \|\theta-\vartheta\|_{L^2(\Om)}
,\spa 
\forall\ \theta,\vartheta\in \AA
.
$$
In addition, the same holds whenever $\theta,\vartheta$ belong to the inertial manifold $\MM$ for some other constant $C_\MM>0$.
\end{Prop}


\begin{Proof}{Proposition \ref{PropReversePoin}}
Let $(e_j)$ be the orthonormal basis of eigenfunctions of the Laplacian $-\Delta$ with corresponding eigenvalues $\lambda_j\to \infty$, and let $P_n$ be the corresponding orthogonal projectors $P_n$ from (\ref{projn}). Theorem \ref{TheorManifold} and the definition (\ref{defInerMan}) of $\MM$ entail that there exists $n\in\N$ and a map 
$$
\Phi : P_n[L^2(\Om)] \to Q_n[L^2(\Om)]
$$ 
such that any $u\in\MM$, hence also any $u\in\AA\subset \MM$, admits the (orthogonal) decomposition $u=u_n+\Phi(u_n)$ with $u_n=P_n[u]$, and such that 
\begin{equation}\label{AContH2}
\|\Phi(v)-\Phi(w)\|_{H^2(\Om)}
\le 
C\|v-w\|_{L^2(\Om)}
,\spa 
\forall\ v,w\in P_n[L^2(\Om)]
,
\end{equation}
for some universal constant $C>0$. For all $\theta,\vartheta\in\MM$, let us write $\theta = \theta_n + \Phi(\theta_n)$ and $\vartheta = \vartheta_n + \Phi(\vartheta_n)$, where  $\theta_n=P_n[\theta]$ and $\vartheta_n=P_n[\vartheta]$. Then,
\begin{align*}
\|\nabla(\theta-\vartheta)\|_{L^2(\Om)}
&\lesssim 
\|\theta-\vartheta\|_{H^2(\Om)}
\\[2mm]
&\le 
\|\theta_n-\vartheta_n\|_{H^2(\Om)} + \|\Phi(\theta_n)-\Phi(\vartheta_n)\|_{H^2(\Om)}
\\[2mm]
&
\lesssim 
(1 + \lambda_n) \|\theta_n-\vartheta_n\|_{L^2(\Om)} + C \| \theta_n - \vartheta_n\|_{L^2(\Om)}
,
\end{align*}
where we used (\ref{AContH2}) and the equivalence $\|u\|^2_{H^2(\Om)} \asymp \sum_{j\ge 0} (1+\lambda_j^2) \ps{u}{e_j}_{L^2(\Om)}^2$ in the last inequality. Because $\|\theta_n-\vartheta_n\|_{L^2(\Om)}\le \|\theta-\vartheta\|_{L^2(\Om)}$, we obtain the reverse Poincaré inequality 
\begin{equation}\label{revPoincare}
\|\nabla(\theta-\vartheta)\|_{L^2(\Om)}
\lesssim
\|\theta-\vartheta\|_{L^2(\Om)}
,\spa 
\forall\ \theta,\vartheta\in \AA
,
\end{equation}
which concludes the proof.
\end{Proof}

\begin{Corol}\label{TheorStability}
For all $t\ge 0$, there exists a constant $C=C(t,f)>0$ such that, for all $\theta,\vartheta$ belonging to the global attractor $\AA$, or the inertial manifold $\MM$ constructed in Section~\ref{sec:Manifold} below, we have
$$
\|\theta-\vartheta\|_{L^2(\Om)} 
\le 
C \|u_{\theta}(t)-u_{\vartheta}(t)\|_{L^2(\Om)}
.
$$
In addition, the constant $C$ can be taken as $C=\exp(c_1 t^2+c_2t)$ for some $c_1,c_2>0$ when $\theta,\vartheta\in\MM$ and as $C=\exp(ct)$ when $\theta,\vartheta\in\AA$.
\end{Corol}

\begin{Proof}{Corollary \ref{TheorStability}}
We apply Proposition \ref{PropStability}. We already observed in Proposition \ref{PropReversePoin} that the global attractor $\AA$ and the inertial manifold $\MM$ satisfy a reverse Poincaré inequality. Thus, it remains to establish boundedness in $L^\infty(\Om)$ uniformly in $\theta\in \AA$ (or $\theta\in\MM$) and in $t\in [0,\infty)$. This follows from Proposition~\ref{prop:SupBound} below, combined with boundedness of $\MM$ (hence also $\AA$) in $H^2(\Om)$ and the Sobolev embedding $H^2(\Om)\embed L^\infty(\Om)$.
\end{Proof}

\begin{Theor}\label{TheorContractTheta}
    Let $U$ be a uniform $L^\infty(\Om)$-bound on $\AA$ and let $n\in\N$ be as in Theorem~\ref{TheorManifold}. Then, there exists a sequence of $\AA$-valued estimators $(\hat \theta_N)$ such that for all $N\ge 1$ 
    $$
    \sup_{\theta_0\in\AA} \E_{\theta_0}\big[\|\hat \theta_N - \theta_0\|^2_{L^2(\Om)}\big]
    \le 
    \frac{8b^2n(U^2+1)e^{ct}}{N}
    ,
    $$
    for some $c>0$ and $b>0$ as in (\ref{eq:EquivPnL2}), depending only on $f$.
    In addition, for all $\tau\ge 0$ we have
    $$
    \sup_{\theta_0\in\AA} 
    \E_{\theta_0}\big[ \|u_{\hat \theta_N}(\tau)-u_{\theta_0}(\tau)\|^2_{L^2(\Om)}\big] 
    \le 
    \frac{8b^2n(U^2+1)e^{c|\tau-t|}}{N} 
    .
    $$
    
\end{Theor}


\begin{Proof}{Theorem \ref{TheorContractTheta}}
    Let $(e_j)_{j\ge 0}$ be the orthonormal basis of $L^2(\Om)$ with corresponding projector $P_n$ from (\ref{projn}). Fix $\theta_0\in\AA$ and let $(X_i,Y_i)_{i=1}^N$ be data obtained through (\ref{model}). Define
    $$
    \hat u_j 
    \equiv 
    \frac{1}{N} \sum_{i=1}^N Y_i e_j(X_i) 
    ,\spa 
    j\ge 0
    .
    $$
    Then $\hat u_j$ is an unbiased estimator of $u_j\equiv \ps{u_{\theta_0}(t)}{e_j}_{L^2(\Om)}$, \ie $\E[\hat u_j] = u_j$, since the design $(X_i)_{i\le N}$ is uniform and  the errors $\ve_i$ in (\ref{model}) are centred and independent from the $X_i$'s. Fix $n\in\N$ as in Theorem~\ref{TheorManifold} and define the estimator of $P_n[u_{\theta_0}(t)]$
    $$
    \hat u_t 
    \equiv 
    \sum_{j=0}^{n-1} \hat u_j e_j 
    \in P_n[L^2(\Om)]
    .
    $$
    Straightforward calculations provide 
    \begin{eqnarray*}
    \lefteqn{
    \hspace{-20mm}
    \E_{\theta_0}\big[\|\hat u_t - P_n[u_{\theta_0}(t)]\|^2_{L^2(\Om)}\big]
    =
    \frac1N \sum_{j=0}^{n-1} \Big\{ 
    {\rm Var}[u_{\theta_0}(t,X_1)e_j(X_1)] + \E[\ve_1^2 e_j(X_1)^2]
    \Big\} 
    }
    \\[1mm]
    &&\le 
    \frac1N \sum_{j=0}^{n-1} \Big\{ 
    \E\big[|u_{\theta_0}(t,X_1)e_j(X_1)|^2\big] + 1 \Big\} 
    \le 
    \frac{n(U^2 + 1)}{N}
    ,
    \end{eqnarray*}
    where we used independence of $\ve_1$ and $X_1$ and the fact that $\E[\ve^2]=1$ and $\E[e_j(X_1)^2] = \|e_j\|^2_{L^2(\Om)}=1$. Now fix a countable dense subset $(a_k)$ of $\AA$ and let 
    $$
    m_N 
    \equiv 
    \inf_{k} \|\hat u_t - P_n[u_{a_k}(t)]\|_{L^2(\Om)}
    .
    $$
    Then, we have 
    $$
    m_N 
    \le 
    \|\hat u_t - P_n[u_{\theta_0}(t)]\|_{L^2(\Om)}
    ,
    $$
    since $(a_k)$ is dense in $\AA$, $\theta_0\in\AA$, and $\theta \mapsto u_{\theta}(t)$ is continuous from $L^2(\Om)$ to $L^2(\Om)$ by virtue of Proposition~\ref{PropSolExist}. Fix an arbitrary positive sequence $(\ve_N)\to 0$ and let $k_N$ be the first integer such that 
    $$
    \|\hat u_t - P_n[u_{a_{k_N}}(t)]\|^2_{L^2(\Om)}
    \le 
    m_N^2 + \ve_N 
    .
    $$
    Define the estimator $\hat \theta_N \equiv a_{k_N}$. Then, for $e^{at}$ the constant featuring in Corollary~\ref{TheorStability} and $b$ the constant from (\ref{eq:EquivPnL2}), we have 
    \begin{eqnarray*}
        \lefteqn{
        \|\hat \theta_N - \theta_0\|^2_{L^2(\Om)}
        =
        \|a_{k_N} - \theta_0\|^2_{L^2(\Om)}
        \le
        e^{2at}\|u_{a_{k_N}}(t) - u_{\theta_0}(t)\|^2_{L^2(\Om)}
        }
        \\[2mm]
        &&
        \le b^2 e^{2at}\|P_n[u_{a_{k_N}}(t) - u_{\theta_0}(t)]\|^2_{L^2(\Om)}
        \\[2mm]
        &&
        \le
        2 b^2 e^{2at} \Big( \|P_n[u_{a_{k_N}}(t)]-\hat u_t\|^2_{L^2(\Om)} + \|\hat u_t - P_n[u_{\theta_0}(t)]\|^2_{L^2(\Om)} \Big)
        \\[2mm]
        &&
        \le 
        2 b^2 e^{2at} \big(2\|\hat u_t - P_n[u_{\theta_0}(t)]\|^2_{L^2(\Om)} + \ve_N\big)
        .
    \end{eqnarray*}
    It follows that 
    \begin{align}\label{eq:UpBound1}
    \E_{\theta_0}\big[ \|\hat \theta_N - \theta_0\|^2_{L^2(\Om)} \big]
    &\le 
    4 b^2 e^{2at} \E_{\theta_0}\big[ \|\hat u_t - P_n[u_{\theta_0}(t)]\|^2_{L^2(\Om)} \big] + 2 b^2 e^{2at}\ve_N
    \\[2mm]
    &\le 
    \frac{4b^2n(U^2+1)e^{2at}}{N} + 2b^2 e^{2at} \ve_N
    ,
    \end{align}
    uniformly for all $\theta_0\in\AA$. The previous calculations also provide
    $$
    \|u_{\hat \theta_N}(t)-u_{\theta_0}(t)\|^2_{L^2(\Om)}
    \le 
    4b^2 \|\hat u_t - P_n[u_{\theta_0}(t)]\|^2_{L^2(\Om)} + 2b^2 \ve_N
    ,
    $$
    which yields 
    \begin{equation*}
    \E_{\theta_0}\big[ \|u_{\hat \theta_N}(t)-u_{\theta_0}(t)\|^2_{L^2(\Om)}\big] 
    \le 
    \frac{4b^2n(U^2+1)}{N} + 2b^2 \ve_N
    .
    \end{equation*}
    Now, using the Lipschitz bound from Proposition~\ref{PropSolExist} with initial conditions $u_{\hat \theta_N}(t)$ and $u_{\theta_0}(t)$ if $\tau > t$, or the stability estimate from Corollary~\ref{TheorStability} with initial conditions $u_{\hat \theta_N}(\tau)\in\AA$ and $u_{\theta_0}(\tau)\in\AA$ if $\tau<t$, implies 
    $$
    \|u_{\hat\theta_N}(\tau) - u_{\theta_0}(\tau)\|_{L^2(\Om)}
    \le 
    e^{C|\tau-t|} \|u_{\hat \theta_N}(t)-u_{\theta_0}(t)\|_{L^2(\Om)}
    ,
    $$
    for some $C>0$ independent of $t,\tau,\theta_0,\hat\theta_N$. We deduce that 
    \begin{equation}\label{eq:upBound2}
    \E_{\theta_0}\big[ \|u_{\hat \theta_N}(\tau)-u_{\theta_0}(\tau)\|^2_{L^2(\Om)}\big] 
    \le 
    \frac{4b^2n(U^2+1)e^{2C|\tau-t|}}{N} + 2b^2 e^{2C|t-\tau|}\ve_N
    ,
    \end{equation}
    uniformly for all $\theta_0\in\AA$.
    Since the upper bounds in (\ref{eq:UpBound1}) and (\ref{eq:upBound2}) are independent of $\theta_0\in\AA$, and $\ve_N$ was arbitrary, we can choose
    $\ve_N = 1/N^2$ for instance so that adding the resulting inequalities and letting $c\equiv 2\max(a,C)$ completes the proof.
\end{Proof}

\section{Background material}

\subsection{Reaction-diffusion equations, absorbing sets, and the global attractor}

Throughout this section, we will assume that $f:\R\to \R$ is a function in $C^3(\R)$ satisfying (\ref{Condf}) for some $p>2$. We denote 
$$
Q_T=(0,T)\times\Omega
,
$$
and define the mutually dual spaces
$$
X_T \equiv L^2([0,T],H^1(\Omega))\cap L^p(Q_T)
,\spa 
X_T' \equiv L^2([0,T],H^{-1}(\Omega)) + L^q(Q_T)
,
$$
where $q=p/(p-1)\in (1,2)$ is the conjugate exponent to $p$. We also note that the duality bracket for the $X_T$-$X_T'$ pairing is simply the usual $\ps{\cdot}{\cdot}_{L^2([0,T]\times\Om)}$ inner product for smooth enough functions. When $\theta\in L^2(\Om)$, we say that $u$ is a weak solution of (\ref{eq:ReacDiff}) when 
$$
u
\in 
X_T,\spa 
\partial_t u 
\in 
X_T'
,
$$
and, for all $\varphi\in X_T$,
\begin{equation}\label{eq:weak}
\int_0^T \ps{\partial_t u}{\varphi}_{L^2(\Om)}\, dt
+
\int_0^T \ps{\nabla u(t)}{\nabla\varphi(t)}_{L^2(\Om)}\, dt
=
\int_0^T \ps{f(u(t))}{\varphi(t)}_{L^2(\Om)}\, dt
.
\end{equation}
Lemma~\ref{lem:growth} below entails that $f(u)\in L^q(Q_T)$ whenever $u\in L^p(Q_T)$ so that the r.h.s. of (\ref{eq:weak}) is well defined. We say that the solution is strong whenever~(\ref{eq:ReacDiff}) holds as an equality between functions in $L^2(Q_T)$.\medskip




The following result follows from Theorem 1.1 in \cite{Marion1987}; see also Sections 8.3-8.4 in \cite{Robinson2001}. 

\begin{Prop}\label{PropSolExist}
For all $\theta\in L^2(\Om)$ and time horizon $T>0$, the system of equations~(\ref{eq:ReacDiff}) admits a unique weak solution $u_{\theta}\in C^0([0,T], L^2(\Om))$ such that $u_{\theta}\in X_T$. In addition, there exists a constant $c=c(f)>0$ such that 
$$
\|u_{\theta}(t)-u_{\vartheta}(t)\|_{L^2(\Om)}
\le 
e^{ct} \|\theta-\vartheta\|_{L^2(\Om)}
,\spa 
\forall\ t\ge 0,\ \theta,\vartheta\in L^2(\Om) 
.
$$
If $\theta\in H^1(\Om)$, then $u_{\theta}$ is a strong solution and we have $u_{\theta}\in C^0([0,T], H^1(\Om))$ with $u_{\theta}\in L^2([0,T], H^2(\Om))\cap L^\infty([0,T], L^p(\Om))$ for all $T>0$.
\end{Prop}

Classical results in \cite{Robinson2001} and \cite{Marion1987} usually require $\theta\in H^1(\Om)\cap L^p(\Om)$, with $p$ as in (\ref{Condf}), to obtain strong solutions in arbitrary dimension. In our two-dimensional case, however, the Sobolev embedding $H^1(\Om)\embed L^r(\Om)$ for all $r\in [1,\infty)$ makes this extra requirement redundant.\medskip 

 We next show that the dynamical system $S_t$ has an absorbing set in $H^1(\Om)$.

\begin{Prop}\label{PropAbsorbH1}
    There exists a constant $C=C(f)>0$ such that for all $\theta\in L^2(\Om)$ we have $\|u_\theta(t)\|_{H^1(\Om)}\le C$ for all $t\ge t_0$, for some $t_0=t_0(\|\theta\|_{L^2(\Om)})$.
\end{Prop}

\begin{Proof}{Proposition \ref{PropAbsorbH1}}
    The proof follows the same lines as that of Proposition 11.1 and 11.3 in \cite{Robinson2001} once we establish that for some constants $b,c>0$, 
    \begin{equation}\label{eq:PDERob}
    \frac{d}{dt}\|u_\theta\|^2_{L^2(\Om)} + 2c\|u_\theta\|^2_{L^2(\Om)}
    \le 
    2b
    ,\spa 
    \forall\ t>0.
    \end{equation}
 To show (\ref{eq:PDERob}) in our periodic setting, we cannot rely on the standard Poincar\'e inequality as in \cite{Robinson2001}. Instead, let us use the growth condition (\ref{Condf}) on $f$ and that $|\Omega|=1$: As in (11.6) in \cite{Robinson2001} we have
    $$
    \frac12 \frac{d}{dt}\|u_\theta\|^2_{L^2(\Om)} + \|\nabla u_\theta\|^2_{L^2(\Om)} + \alpha_2 \|u_\theta\|^p_{L^p(\Om)}
    \le 
    k
    ,\spa 
    \forall\ t\ge 0 
    .
    $$
    Because $p\ge 2$, H\"older's and Young's inequalities provide 
    $$
    \|u_\theta\|^2_{L^2(\Om)}
    \le  
    \|u_\theta\|^2_{L^p(\Om)}
    \le 
    \frac2p \|u_\theta\|^p_{L^p(\Om)} + \frac{p-2}{p}
    ,
    $$
   using again that $|\Om|=1$. Combining the preceding displays we deduce that 
    $$
    \frac12 \frac{d}{dt}\|u_\theta\|^2_{L^2(\Om)} 
    + 
    \|\nabla u_\theta\|^2_{L^2(\Om)} 
    + 
    \frac{\alpha_2 p}{2} \|u_\theta\|^2_{L^2(\Om)} - \frac{\alpha_2(p-2)}{2} 
    \le 
    k
    ,\spa 
    \forall\ t\ge 0 
    .
    $$
    Since $\alpha_2>0$, then dropping the term $\|\nabla u_\theta\|_{L^2(\Om)}$ in the last display yields (\ref{eq:PDERob}) hence concludes the proof.
\end{Proof}

From the previous proposition we can deduce:

\begin{Prop}\label{PropAttractor}
Let $S_t\theta=u_\theta(t)$ be the solution operator for (\ref{eq:ReacDiff}) with $f$ satisfying (\ref{Condf}). Then, the semidynamical system $(L^2(\Omega),\{S_t\}_{t \geq 0})$ admits a global attractor $\AA$ in the sense of Definition~10.4 in \cite{Robinson2001}; moreover $\mathcal A$ is bounded in $H^2(\Omega)$, and given by formula (\ref{lazydimitri}). 
\end{Prop}
\begin{proof}
That a global attractor in the sense of Definition 10.4 in \cite{Robinson2001} exists follows from Theorem 10.5 in \cite{Robinson2001} since the semigroup $\{S_t\}_{t\ge 0}$ is dissipative in the sense of Definition 10.2 there, with  absorbing (compact) $H^1(\Om)$-ball by virtue of Proposition~\ref{PropAbsorbH1}. That this global attractor coincides with our definition in (\ref{lazydimitri}) follows from eq.~(10.10) in \cite{Robinson2001}. The proof of $H^2(\Om)$-boundedness follows as in the proof of Theorem 11.7 in \cite{Robinson2001}.
\end{proof}


It follows in particular that $\AA$ is invariant under the dynamics:
\begin{equation}\label{invA}
    u_\theta(t)\in \AA
    ,\spa 
    \forall\ t\ge 0,\ \forall\ \theta\in\AA,
\end{equation}
and independent of the choices of the absorbing set.

\subsection{Uniform-in-time $L^\infty$ bound}


The following chain rule for positive parts is a standard result and can be found, e.g., in Lemma~3.3 of \cite{Wachsmuth2016} after only routine adaptations.

\begin{Lem}\label{lem:chain}
Let $w\in X_T$ with $\partial_t w\in X_T'$. Then $w^+\equiv\max(w,0)\in X_T$, $\nabla w^+= 1_{\{w>0\}}\nabla w$, and the map $t\mapsto\|w^+(t)\|^2_{L^2(\Om)}$ is absolutely continuous on $[0,T]$ with
$$
\tfrac12\|w^+(t)\|_{L^2(\Om)}^2-\tfrac12\|w^+(s)\|_{L^2(\Om)}^2=\int_s^t \ps{\partial_t w(\tau)}{w^+(\tau)}_{L^2(\Om)}\,d\tau
,\spa 
\forall\ 0\le s\le t\le T
.
$$
\end{Lem}

\begin{Lem}\label{lem:growth}
Let $f$ be as in (\ref{Condf}). (i) There is a constant $C_f$ such that $|f(s)|\le C_f\,(1+|s|^{p-1})$ for all $s\in\R$. (ii) For $K_0=(k/\alpha_2)^{1/p}$, then $f(s)\le0$ for all $s\ge K_0$ and $f(s)\ge0$ for all $s\le-K_0$.
\end{Lem}

\begin{Proof}{Lemma \ref{lem:growth}}
(i) For $s\ge1$, dividing \eqref{Condf} by $s>0$ gives
$$
-k/s-\alpha_1 s^{p-1}\le f(s)\le k/s-\alpha_2 s^{p-1}
,
$$
hence $|f(s)|\le k+\alpha_1 s^{p-1}$. For $s\le-1$, dividing by $s<0$ reverses the inequalities and yields the same bound $|f(s)|\le k+\alpha_1|s|^{p-1}$.
On $[-1,1]$, $f$ is bounded since it is continuous, so that taking $C_f=k+\alpha_1+\max_{t\in[-1,1]}|f(t)|$ yields the desired inequality.\medskip

(ii) If $s\ge K_0>0$ then $\alpha_2 s^p\ge k$, so that \eqref{Condf} entails that $f(s)=f(s)s/s\le(k-\alpha_2 s^p)/s\le0$.
If $s\le-K_0<0$ then again $k-\alpha_2|s|^p\le0$, and dividing $f(s)s\le k-\alpha_2|s|^p$ by $s<0$ gives $f(s)\ge(k-\alpha_2|s|^p)/s\ge0$. If $k=0$ then $K_0=0$ and \eqref{Condf} entails that $f(s)\le 0$ for $s\ge 0$ and $f(s)\ge0$ for $s\le0$, as was to be proved, hence concludes the proof.
\end{Proof}

We now state a standard comparison principle for the reaction-diffusion equation (\ref{eq:ReacDiff}). For the sake of completeness, we provide a self-contained proof by usual comparison arguments together with the absolute continuity Lemma~\ref{lem:chain}.

\begin{Lem}\label{prop:comp}
Let $f$ be as in (\ref{Condf}), $u\in X_T$ be a weak solution of \eqref{eq:ReacDiff}, and let $R\in\R$.
\begin{enumerate}
\item[(i)] If $f(R)\le0$ and $\theta(x)\le R$ for almost every $x\in \Omega$, then $u_\theta(t,x)\le R$ for almost every $x\in \Omega$ and every $t\ge0$.
\item[(ii)] If $f(R)\ge0$ and $\theta(x)\ge R$ for almost every $x\in\Omega$, then $u_\theta(t,x)\ge R$ for almost every $x\in \Omega$ and every $t\ge0$.
\end{enumerate}
\end{Lem}

\begin{Proof}{Lemma \ref{prop:comp}}
Throughout the proof, we let $u\equiv u_\theta$ for simplicity.
(i) Letting $w=u-R$, we have $w\in X_T$, $\partial_t w=\partial_t u\in X_T'$ and $\nabla w=\nabla u$. Lemma~\ref{lem:chain} thus entails that $w^+\equiv \max(w,0)\in X_T$ so that, for any $t\in(0,T]$, we may take $\varphi=w^+ 1_{(0,t)}$ in \eqref{eq:weak}. Using Lemma~\ref{lem:chain} on the l.h.s. of (\ref{eq:weak}) together with $\nabla u\cdot\nabla w^+=\nabla w\cdot\nabla w^+=|\nabla w^+|^2$ in $\R^2$, we obtain
\begin{equation}\label{eq:energy}
\frac12\|w^+(t)\|^2_{L^2(\Om)} - \frac12\|w^+(0)\|^2_{L^2(\Om)}
+
\int_0^t\|\nabla w^+(s)\|^2_{L^2(\Om)}\, ds
=
\int_0^t \ps{f(u(s))}{w^+(s)}_{L^2(\Om)}\, ds
.
\end{equation}
Since $f(R)\le0$ and $w^+\ge0$, we have 
$$
\int_0^t \ps{f(u(s))}{w^+(s)}_{L^2(\Om)}\, ds
\le 
\int_0^t \ps{f(u(s))-f(R)}{w^+(s)}_{L^2(\Om)}\, ds
.
$$
On the set $\{w>0\}$ we have $u=R+w$ with $w>0$, so the fact that $f'\le L$ (see (\ref{Condf})) entails that
$$
f(u(t))-f(R)
=
f(R+w(t))-f(R)
=
w(t)\int_0^1 f'(R+ s w(t))\, ds
\le L w(t)
.
$$
Consequently, dropping the nonnegative gradient term in \eqref{eq:energy}, we find
$$
\|w^+(t)\|^2_{L^2(\Om)}
\le
\|w^+(0)\|^2_{L^2(\Om)}
+
2L\int_0^t\|w^+(s)\|^2_{L^2(\Om)}\, ds
,\spa 
\forall\ t\in [0,T]
.
$$
Gronwall's inequality thus yields $\|w^+(t)\|^2_{L^2(\Om)}\le e^{2L_+t}\|w^+(0)\|^2_{L^2(\Om)}$, where $L_+\equiv \max(L,0)$. But $w^+(0)=(\theta-R)^+=0$ because $\theta\le R$. We deduce that $w^+(t)=0$, i.e.\ $u(t)\le R$ almost everywhere in $\Omega$ for all $t\in[0,T]$. The conclusion for (i) now follows from the fact that $T>0$ was arbitrary.\medskip

(ii) This part follows by applying the same argument to $w=R-u$. Then $\partial_t w=-\partial_t u$, $\nabla w=-\nabla u$, and \eqref{eq:weak} tested against $\varphi=w^+ 1_{(0,t)}$ provides
\begin{eqnarray*}
\lefteqn{
\frac12\|w^+(t)\|^2_{L^2(\Om)}
-
\frac12\|w^+(0)\|^2_{L^2(\Om)}
+
\int_0^t \|\nabla w^+(s)\|^2_{L^2(\Om)}\, ds
}
\\[2mm]
&&
=
\int_0^t\ps{-f(u(s))}{w^+(s)}_{L^2(\Om)}\, ds
\le
\int_0^t \ps{f(R)-f(u(s))}{w^+(s)}_{L^2(\Om)}\, ds
,
\end{eqnarray*}
using $f(R)\ge 0$. On $\{w>0\}$, we have $u=R-w$ and 
$$
f(R)-f(u(t))
=
w(t)\int_0^1 f'(R- s w(t))\, ds
\le 
Lw(t)
,
$$
so that the conclusion holds by following the same reasoning as in (i).
\end{Proof}


\begin{Prop}\label{prop:SupBound}
Let $f$ be as in (\ref{Condf}). Let $\theta\in L^\infty(\Omega)$ and let $u_\theta$ be the weak solution of \eqref{eq:ReacDiff}. Then, with $K_0=(k/\alpha_2)^{1/p}$,
$$
\|u_\theta(t)\|_{L^\infty(\Omega)}
\le
\max\big(\|\theta\|_{L^\infty(\Omega)},\,K_0\big)
,\spa 
\forall\ t\ge0 
.
$$
\end{Prop}

\begin{Proof}{Proposition \ref{prop:SupBound}}
Set $R=\max(\|\theta\|_{L^\infty(\Om)},K_0)$. Since $R\ge K_0$, Lemma~\ref{lem:growth} entails that $f(R)\le0$ and $f(-R)\ge0$. In addition, we have $-R\le\theta\le R$ almost everywhere in $\Omega$. Lemma~\ref{prop:comp} thus yields $u_\theta(t)\le R$ and $u_\theta(t)\ge -R$ almost everywhere on $\Om$ for all $t\ge0$. We deduce that $\|u_\theta(t)\|_{L^\infty(\Om)}\le R$ for all $t\ge0$, which concludes the proof. 
\end{Proof}

\subsection{The inertial manifold}\label{sec:Manifold}

While the global attractor $\mathcal A$ describes the precise limiting states of our reaction diffusion equation, it has a complicated analytical structure. Instead we shall work with a slightly larger manifold $\mathcal M$ that contains $\mathcal A$, that shares many properties with $\mathcal A$ but is analytically more tractable (once its existence is shown). \smallskip

To construct an inertial manifold $\mathcal M$ containing the global attractor $\AA$ from our PDE we will follow arguments developed in \cite{Foias1989}, and this requires some auxiliary results which we provide now.\smallskip

First we will need to upgrade Proposition \ref{PropAbsorbH1} to establish that the semidynamical system $(H^2(\Om), \{ S_t : t\ge 0\})$ admits a bounded $H^2(\Om)$-absorbing set -- note that this is a stronger requirement than the boundedness of $\AA$ in $H^2(\Om)$, which will be needed in the proof of Theorem~\ref{TheorManifold} below.

\begin{Theor}\label{TheorAbsorbH2}
    There exists a constant $C=C(f)>0$ such that for all $\theta\in H^1(\Om)$ we have $\|u_\theta(t)\|_{H^2(\Om)}\le C$ for all $t\ge \tilde t$, for some $\tilde t=\tilde t(\|\theta\|_{H^1(\Om)})$.
\end{Theor}

\begin{Proof}{Theorem \ref{TheorAbsorbH2}}
Proposition \ref{PropAbsorbH1} entails that there exists a constant $c_1>0$ independent of $\theta$ such that
\begin{equation} \label{gradest}
\|u_\theta(t)\|_{L^2(\Om)}+\|\nabla u_\theta(t)\|_{L^2(\Om)} \le c_1 
,\spa 
\forall\ t\ge t_0\equiv t_0(\|\theta\|_{H^1(\Om)})
,
\end{equation}
where $t_0$ only depends on $\|\theta\|_{H^1(\Om)}$. Taking the $L^2(\Om)$-inner product of (\ref{eq:ReacDiff}) with $-\Delta u_\theta$ yields 
$$
\frac12 \frac{d}{dt}\|\nabla u_\theta\|^2_{L^2(\Om)}
+
\|\Delta u_\theta\|^2_{L^2(\Om)}
=
-\ps{f(u_\theta)}{\Delta u_\theta}_{L^2(\Om)}
\lesssim
\|\nabla u_\theta\|^2_{L^2(\Om)},
$$
arguing as in the beginning of the proof of Theorem~8.5~in \cite{Robinson2001}, using also the second part of the hypothesis (\ref{Condf}) for $f$.
Integrating the last display between $t$ and $t+1$ and using (\ref{gradest}) yields 
\begin{equation} \label{laplaceest}
\|\nabla u_\theta(t+1)\|^2_{L^2(\Om)}
+
\int_{t}^{t+1} \|\Delta u_\theta(s)\|^2_{L^2(\Om)}\, ds
\lesssim 
c_1^2 + \int_{t}^{t+1} \|\nabla u_\theta(s)\|^2_{L^2(\Om)}\, ds
\le 2c_1^2 \equiv \ c_2
,\spa 
\forall\ t\ge t_0
\end{equation}
for some $c_2>0$ independent of $\theta$. In particular, the mean value theorem for integrals entails that there exists $s\in [t_0, t_0+1]$ such that $\|\Delta u_\theta(s)\|^2_{L^2(\Om)} \lesssim c_2$. Therefore, the Sobolev embedding $H^2(\Om)\embed L^\infty(\Om)$ yields 
$$
\|u_\theta(s)\|^2_{L^\infty(\Om)}
\lesssim 
\|u_\theta(s)\|^2_{H^2(\Om)}
\lesssim 
\|u_\theta(s)\|^2_{L^2(\Om)}
+
\|\Delta u_\theta(s)\|^2_{L^2(\Om)} 
\lesssim 
c_1^2+c_2
.
$$
Proposition~\ref{prop:SupBound} with initial condition $u_\theta(s)$ provides $\|u_\theta(t)\|_{L^\infty(\Om)}
\le 
\max(\|u_\theta(s)\|_{L^\infty(\Om)}, K_0)$ for all $t\ge s$ and some universal constant $K_0>0$. It follows that
$$
\|u_\theta(t)\|_{L^\infty(\Om)}
\le 
c_3 
,\spa 
\forall\ t\ge t_0+1 \equiv t_1(\|\theta\|_{H^1(\Om)})
,
$$
with $c_3>0$ a constant independent of $\theta$.
Now, formally taking the inner product of (\ref{eq:ReacDiff}) with $\Delta^2 u_\theta$ (this can be justified through a Galerkin procedure) yields 
\begin{align*}
\frac12 \frac{d}{dt}\|\Delta u_\theta\|^2_{L^2(\Om)}  -\langle \Delta u_\theta, \Delta^2 u_\theta \rangle_{L^2(\Om)} &= \langle f(u_\theta), \Delta^2 u_\theta \rangle_{L^2(\Om)}
\end{align*}
so that by integration by parts, and the Cauchy-Schwarz and Young inequalities, 
$$\frac12 \frac{d}{dt}\|\Delta u_\theta\|^2_{L^2(\Om)}
+
\|\nabla \Delta u_\theta\|^2_{L^2(\Om)} \\
= -\ps{\nabla f(u_\theta)}{\nabla \Delta u_\theta}_{L^2(\Om)} \le 
\frac{1}{2}\|\nabla (f(u_\theta))\|^2_{L^2(\Om)} + \frac12 \|\nabla \Delta u_\theta\|^2_{L^2(\Om)}.$$
Rearranging and dropping the $\|\nabla \Delta u_\theta\|_{L^2(\Om)}$ term yields
$$
\frac{d}{dt}\|\Delta u_\theta\|^2_{L^2(\Om)}
\le
\|\nabla(f(u_\theta))\|^2_{L^2(\Om)}
.
$$
A routine density argument as in Lemma~\ref{LemgH2} below yields the $\nabla (f(u_\theta))=f'(u_\theta) \nabla u_\theta$ in the weak sense in $L^2(\Om)$ which, together with the $L^\infty(\Om)$ bound on $u_\theta(t)$ above, provides 
$$
\|\nabla(f(u_\theta))\|_{L^2(\Om)}
=
\|f'(u_\theta) \nabla u_\theta\|_{L^2(\Om)}
\le
\|f'(u_\theta)\|_{L^\infty(\Om)} \|\nabla u_\theta\|_{L^2(\Om)}
\le 
\|f'\|_{L^\infty([-c_3, c_3])} \|\nabla u_\theta\|_{L^2(\Om)}
,
$$
for all $t\ge t_1$.
This yields 
$$
\frac{d}{dt}\|\Delta u_\theta\|^2_{L^2(\Om)}
\lesssim 
\|\nabla u_\theta\|^2_{L^2(\Om)}
.
$$
Now integrating the last display between $s$ and $t$, $s \le t$ provides 
$$
\|\Delta u_\theta(t)\|^2_{L^2(\Om)}
\lesssim 
\|\Delta u_\theta(s)\|^2_{L^2(\Om)}
+
\int_s^t \|\nabla u_\theta(\tau)\|^2_{L^2(\Om)}\, d\tau 
.
$$
Integrating in $s$ between $t-1$ and $t$, using (\ref{gradest}) and (\ref{laplaceest}), and applying Fubini's theorem yields 
$$
\|\Delta u_\theta(t)\|^2_{L^2(\Om)}
\lesssim 
\int_{t-1}^t \|\Delta u_\theta(s)\|^2_{L^2(\Om)}\, ds
+
\int  _{t-1}^t \|\nabla u_\theta(s)\|^2_{L^2(\Om)}\,ds
\le c_1^2+c_2
,~~
\forall\ t\ge 1+\max\{t_0, t_1\}
.
$$
Combining this with (\ref{gradest}) entails that 
$$
\|u_\theta(t)\|_{H^2(\Om)}
\lesssim \|u_\theta(t)\|_{L^2(\Om)} + \|\Delta u_\theta(t)\|_{L^2(\Om)} \le c$$
for all $t$ large enough depending only on $\|\theta\|_{H^1(\Om)}$ and some $c>0$, which yields the conclusion.
\end{Proof}

\begin{Lem}\label{LemgH2}
    Let $h\in C^2(\R)$ and $u\in H^2(\Om)$. Then, $h(u)\in H^2(\Om)$, and 
    for some universal constant $c>0$, we have
    $$
    \|h(u)\|_{H^2(\Om)}
    \lesssim 
    M_h(\|u\|_{H^2(\Om)})
    ,
    $$
    where
    $ 
    M_h(s) 
    \equiv
    \|h\|_{L^\infty([-cs, cs])} 
    +
    \|h'\|_{L^\infty([-cs, cs])}s + \|h''\|_{L^\infty([-cs,cs])} s^2
    $
    for all $s\ge 0$.
\end{Lem}

\begin{Proof}{Lemma \ref{LemgH2}}
First note that $u\in H^2(\Om)\embed L^\infty(\Om)$ by the Sobolev embedding, so that we can write
\begin{equation}\label{Richardwantsnumbers}
|h(u)| 
\le 
|h(u)| 1_{\{|u|\le c \|u\|_{H^2(\Om)}\}}
\le 
\|h\|_{L^\infty([-c\|u\|_{H^2(\Om)}, c\|u\|_{H^2(\Om)}])}
<
\infty
,
\end{equation}
where $c>0$ is any number such that $\|u\|_{L^\infty(\Om)}\le c \|u\|_{H^2(\Om)}$, and noting that $h$ is continuous over~$\R$. It follows that $h(u)\in L^\infty(\Om)$ hence also $h(u)\in L^2(\Om)$ since $\Om$ is bounded. Thus, it remains to establish that $\Delta(h(u))\in L^2(\Om)$. We argue by density: take $(u_\ve)\subset C^\infty(\Om)$ such that $u_\ve \to u$ in $H^2(\Om)$, hence also in $L^\infty(\Om)$ and in $W^{1,4}(\Om)$ by the Sobolev embedding $H^2(\Om)\embed W^{1,4}(\Om)$. The usual chain rule provides 
\begin{equation}\label{eq:LapConvol}
\Delta(h(u_\ve)) 
=
h''(u_\ve) |\nabla u_\ve|^2 + h'(u_\ve) \Delta u_\ve 
.
\end{equation}
Let $M=\|u\|_{L^\infty(\Om)}$. Since $u_\ve\to u$ in $L^\infty(\Om)$, then $\|u_\ve\|_{L^\infty(\Om)}\le M+1$ eventually. Because $h'$ and $h''$ are uniformly continuous over $[-(M+1), M+1]$, the r.h.s. of (\ref{eq:LapConvol}) converges to $h''(u)|\nabla u|^2 + h'(u)\Delta u$ in $L^2(\Om)$. Similarly, the uniform continuity of $h$ yields $h(u_\ve)\to h(u)$ in $L^2(\Om)$, to the effect that the l.h.s. of (\ref{eq:LapConvol}) converges to $\Delta(h(u))$ in the sense of distributions (in the topology of $\DD(\Om)')$. Consequently, the distributional Laplacian $\Delta(h(u))$ is given by $h''(u)|\nabla u|^2 + h'(u)\Delta u\in L^2(\Om)$, which yields $\Delta(h(u))\in L^2(\Om)$, to the effect that $h(u)\in H^2(\Om)$. Finally, the expression for $\Delta(h(u))$ implies that
\begin{align*}
&\|h(u)\|_{H^2(\Om)}
\\[2mm]
\lesssim~~
&\|h(u)\|_{L^2(\Om)} + \|\Delta (h(u))\|_{L^2(\Om)}
\\[2mm]
\lesssim~~
&\|h\|_{L^\infty([-c\|u\|_{H^2(\Om)}, c\|u\|_{H^2(\Om)}])}
+
\|h''(u)\|_{L^\infty(\Om)} \|\nabla u\|^2_{L^4(\Om)} 
+ 
\|h'(u)\|_{L^\infty(\Om)} \|\Delta u\|_{L^2(\Om)}
\\[2mm]
\lesssim~~
&\|h\|_{L^\infty([-c\|u\|_{H^2(\Om)}, c\|u\|_{H^2(\Om)}])}
+
\|h'(u)\|_{L^\infty(\Om)} \|u\|_{H^2(\Om)} 
+
\|h''(u)\|_{L^\infty(\Om)}  \|u\|^2_{H^2(\Om)} 
,
\end{align*}
by virtue of the Sobolev embedding $H^1(\Om)\embed L^4(\Om)$ applied to $\nabla u$. Arguing as in (\ref{Richardwantsnumbers}) to bound $\|h''(u)\|_{L^\infty(\Om)}$ and $\|h'(u)\|_{L^\infty(\Om)}$ thus yields
$$
\|h(u)\|_{H^2(\Om)}
\lesssim 
M_h(\|u\|_{H^2(\Om)})
,
$$
with $M_h$ as in the statement.
\end{Proof}

\begin{Lem}\label{LemDiffR} 
Let $g\in C^3(\R)$ and  $u\in H^2(\Om)$. Then we have
$$
\|g'(u)v\|_{L^2(\Om)}
\lesssim 
M_0(\|u\|_{H^2(\Om)}) \|v\|_{L^2(\Om)} 
,\spa 
v\in L^2(\Om)
,
$$
where $M_0(s) 
\equiv
\|g'\|_{L^\infty([-cs, cs])}$ for all $s\ge 0$ and some universal $c>0$, and
$$
\|g'(u)v\|_{H^2(\Om)}
\lesssim 
M_1(\|u\|_{H^2(\Om)}) \|v\|_{H^2(\Om)}
,\spa 
v\in H^2(\Om)
,
$$
where $M_1(s)
\equiv 
\|g'\|_{L^\infty([-cs, cs])} 
+
\|g''\|_{L^\infty([-cs, cs])}s + \|g'''\|_{L^\infty([-cs,cs])}s^2$ for all $s\ge 0$.
\end{Lem}

\begin{Proof}{Lemma \ref{LemDiffR}}
For the first inequality, we bound
$$
\|g'(u)v\|_{L^2(\Om)}
\le 
\|g'(u)\|_{L^\infty(\Om)} \|v\|_{L^2(\Om)}
,
$$
and the first inequality in the proof of Lemma \ref{LemgH2} applied to $h\equiv g'$ yields the desired bound. For the second inequality, we use the multiplier inequality for Sobolev norms to the effect that 
$$
\|g'(u)v\|_{H^2(\Om)}
\lesssim 
\|g'(u)\|_{H^2(\Om)} \|v\|_{H^2(\Om)}
,
$$
and Lemma \ref{LemgH2} applied to $h\equiv g'$ concludes the proof.
\end{Proof}






The negative Laplacian $-\Delta$ has eigenfunctions $e_j$ forming an orthonormal basis of $L^2(\Omega)$ with corresponding eigenvalues $\lambda_j$ satisfying $0=\lambda_0 \le \ldots \le \lambda_{j-1} \le \lambda_j \asymp j$ as $j \to \infty$. 
For any $n\in\N$, let $P_n$ denote the projection of $u \in L^2(\Omega)$ onto the first $n$ eigenfunctions of $-\Delta$, \ie 
\begin{equation}\label{projn}
P_n[u]
\equiv 
\sum_{j=0}^{n-1} \ps{u}{e_j}_{L^2(\Om)} e_j 
,
\end{equation}
and $Q_n[u]=u-P_n[u]$ the projection onto its orthogonal complement.\smallskip


\begin{Theor} \label{TheorManifold}
Let $\mathcal A$ be the global attractor from Proposition \ref{PropAttractor}. There exists $n\in\N$ and a map $\Phi : P_n[L^2(\Om)]\to Q_n[L^2(\Om)]\cap H^2(\Om)$ such that (i) any $u\in\AA$ writes as $u=u_n + \Phi(u_n)$ with $u_n=P_n[u]$, and (ii) there exists a constant $C>0$ such that
\begin{equation}\label{AContH2State}
\|\Phi(v)-\Phi(w)\|_{H^2(\Om)}
\le 
C\|v-w\|_{L^2(\Om)}
,\spa 
\forall\ v,w\in P_n[L^2(\Om)]
.
\end{equation}
\end{Theor}

Let $\rho>0$ be an $H^2(\Om)$-bound on $\AA$ (see Proposition~\ref{PropAttractor}) and let $B_{H^2}(8\rho)$ be the centred $H^2(\Om)$-ball with radius $8\rho$. Theorem \ref{TheorManifold} entails the existence of an $L^2(\Om)$-Lipschitz map $\Phi:P_n[L^2(\Om)]\to Q_n[L^2(\Om)]$ such that any $u\in\AA$ writes (uniquely) as $u=\Phi(u_n)+u_n$ where $u_n\equiv P_n[u]$. In addition, $\Phi$ can be taken to be supported on the $n$-dimensional ellipsoid \begin{equation}\label{defDPhi}
D(\Phi) 
\equiv 
P_n[L^2(\Om)]\cap B_{H^2}(8\rho)
.
\end{equation}
Next, the map $G: D(\Phi)\to L^2(\Om)$ defined as 
$$
G(v) 
=
v+\Phi(v) 
,\spa 
v\in D(\Phi)
$$
is also $L^2(\Om)$-$L^2(\Om)$ Lipschitz. We thus define an `inertial manifold' for the reaction-diffusion system as the `graph' of $\Phi$
\begin{equation}\label{defInerMan}
    \MM  
    \equiv
    G(D(\Phi))
    =
    \Big\{ v + \Phi(v) : v\in P_n[L^2(\Om)],~ \|v\|_{H^2(\Om)}\le 8\rho
    \Big\} 
    .
\end{equation}
This is a bounded subset of $H^2(\Om)$ and contains $\AA$ as a subset. \medskip

Observe that Theorem \ref{TheorManifold} does not imply that any $u\in L^2(\Om)$ can be decomposed as $u = P_n[u] + \Phi(P_n[u])$, otherwise the $H^2(\Om)$-Lipschitz continuity in Theorem~\ref{TheorManifold}(ii) could not hold. In fact, $\Phi$ need not even be linear, and one can further show that $\Phi(v)=0$ when $\|v\|_{H^2(\Om)}>8\rho$.\medskip

It follows from Theorem~\ref{TheorManifold} that 
\begin{equation}\label{eq:EquivH2L2}
\|u-v\|_{H^2(\Om)}
\le 
a \|u-v\|_{L^2(\Om)}
,\spa 
\forall\ u,v\in\MM 
,
\end{equation}
for some $a>0$ depending only on $f$ (see the proof of Proposition~\ref{PropReversePoin}), and straightforward calculations further provide
\begin{equation}\label{eq:EquivPnL2}
\|u-v\|_{L^2(\Om)}
\le 
b \|P_n[u-v]\|_{L^2(\Om)}
,\spa 
\forall\ u,v\in\MM 
,
\end{equation}
for some $b>0$ also depending only on $f$.


\begin{Proof}{Theorem \ref{TheorManifold}}
The existence of a map $\Phi$ with the required properties follows from Theorem 3.1 in \cite{Foias1989}. First, we rewrite the reaction-diffusion system (\ref{eq:ReacDiff}) as 
\begin{equation}\label{PDEA}
\frac{du}{dt} + A[u] = g(u)
,
\end{equation}
where $A[u] \equiv (\delta {\rm Id} - \Delta)u$ with $\delta>0$, and $g(s) \equiv f(s) + \delta s$. Then, $A$ is linear, self-adjoint on $L^2(\Om)$ and  bijective from $H^2(\Om)$ to $L^2(\Om)$, and $g$ satisfies the same assumptions as $f$ in (\ref{Condf}) for different values for the parameters $k, \alpha_1, \alpha_2$ and $L$. We will now check the conditions of Section~2 in \cite{Foias1989}, with operator $A$ as above, Hilbert space $H\equiv L^2(\Om)$, dense domain $D(A)\equiv H^2(\Om)$, eigenvalues $\sigma_j \equiv \delta + \lambda_j$ such that $0<\sigma_0\le \sigma_1 \le \ldots \sigma_j \asymp j$ as $j\to\infty$ (provided $\sigma_0=\delta>0$), and functional $R[u]\equiv -g(u)$. Since $g\in C^2(\R)$, Lemma~\ref{LemgH2} provides $g(u)\in H^2(\Om)$ for any $u\in H^2(\Om)$, to the effect that $R$ maps $H^2(\Om)$ into $H^2(\Om)$. In addition, $R$ is Fréchet differentiable from $D(A)\equiv H^2(\Om)$ into $D(A^{1-\beta})\equiv H^2(\Om)$ with $\beta=0$; see Proposition~6.4 in \cite{Kon25Minimax}. Since $g\in C^3(\R)$, Lemma \ref{LemDiffR} implies that (2.1a) and (2.1b) in \cite{Foias1989} hold with $\beta=0$, and non-negative and monotone non-decreasing functions $M_0$ and $M_1$ featuring in Lemma \ref{LemDiffR}. Theorem~\ref{TheorAbsorbH2} entails that $(H^2(\Om), \{S_t : t\ge 0\})$ admits an absorbing ball in $H^2(\Om)$, \ie there exists an $H^2(\Om)$-ball $B$ such that the image of any $H^2(\Om)$-ball through $S_t$ is included in $B$ for all $t$ large enough. The sequence $\{\lambda_{k+1}-\lambda_k : k\ge 0\}$ is unbounded (see Section 15.4.2 in \cite{Robinson2001}) so that, for any $C>0$, there exists $n\in\N$ such that 
$$
\sigma_{n}-\sigma_{n-1} 
=
\lambda_{n}-\lambda_{n-1} 
>
C
,
$$
which yields (3.28) in \cite{Foias1989}. Consequently, Theorem 3.1 there yields the existence of a map $\Phi$ as in the statement with $n$ as above.
\end{Proof}

\textbf{Acknowledgements.} The authors would like to thank Soufiane Noubir for various interesting discussions about the content of this article, and further gratefully acknowledge funding from an ERC Advanced Grant (UKRI G116786) and EPSRC programme grant EP/V026259.

\bibliographystyle{agsm}
\bibliography{mybib}

@article{Wachsmuth2016,
  author  = {Wachsmuth, Daniel},
  title   = {The regularity of the positive part of functions in {$L^2(I;H^1(\Omega))\cap H^1(I;H^1(\Omega)^*)$} with applications to parabolic equations},
  journal = {Commentationes Mathematicae Universitatis Carolinae},
  volume  = {57},
  number  = {3},
  pages   = {327--332},
  year    = {2016},
  doi     = {10.14712/1213-7243.2015.168}
}

@article {ABJN25,
    AUTHOR = {Alberti, Giovanni S. and Barnes, Douglas and Jambhale, Aditya
              and Nickl, Richard},
     TITLE = {On low frequency inference for diffusions without the hot
              spots conjecture},
   JOURNAL = {Math. Stat. Learn.},
  FJOURNAL = {Mathematical Statistics and Learning},
    VOLUME = {8},
      YEAR = {2025},
    NUMBER = {3-4},
     PAGES = {305--322},
      ISSN = {2520-2316,2520-2324},
   MRCLASS = {62G05 (35P10 35R30 60H10)},
  MRNUMBER = {4963422},
       DOI = {10.4171/msl/53},
}

@article {N24,
    AUTHOR = {Nickl, Richard},
     TITLE = {Consistent inference for diffusions from low frequency
              measurements},
   JOURNAL = {Ann. Statist.},
  FJOURNAL = {The Annals of Statistics},
    VOLUME = {52},
      YEAR = {2024},
    NUMBER = {2},
     PAGES = {519--549},
      ISSN = {0090-5364,2168-8966},
   MRCLASS = {60H10 (62F15 62G20 62H25 62M15)},
  MRNUMBER = {4744186},
MRREVIEWER = {Robert\ Stelzer},
       DOI = {10.1214/24-aos2357},
}

@book {T97,
    AUTHOR = {Temam, Roger},
     TITLE = {Infinite-dimensional dynamical systems in mechanics and
              physics},
    SERIES = {Applied Mathematical Sciences},
    VOLUME = {68},
   EDITION = {Second},
 PUBLISHER = {Springer-Verlag, New York},
      YEAR = {1997},
     PAGES = {xxii+648},
      ISBN = {0-387-94866-X},
   MRCLASS = {58Fxx (34G20 35K55 35Q30 47H20 58D25 76D05)},
  MRNUMBER = {1441312},
MRREVIEWER = {Bruno\ Scheurer},
       DOI = {10.1007/978-1-4612-0645-3},
}

@book {BV92,
    AUTHOR = {Babin, A. V. and Vishik, M. I.},
     TITLE = {Attractors of evolution equations},
    SERIES = {Studies in Mathematics and its Applications},
    VOLUME = {25},
      NOTE = {Translated and revised from the 1989 Russian original by
              Babin},
 PUBLISHER = {North-Holland Publishing Co., Amsterdam},
      YEAR = {1992},
     PAGES = {x+532},
      ISBN = {0-444-89004-1},
   MRCLASS = {58F12 (35B40 35K22 47H20 47N20 58F39 76D05)},
  MRNUMBER = {1156492},
MRREVIEWER = {Woodford\ W.\ Zachary},
}

@article {BV83,
    AUTHOR = {Babin, A. V. and Vishik, M. I.},
     TITLE = {Attractors of evolution partial differential equations and
              estimates of their dimension},
   JOURNAL = {Uspekhi Mat. Nauk},
  FJOURNAL = {Akademiya Nauk SSSR i Moskovskoe Matematicheskoe Obshchestvo.
              Uspekhi Matematicheskikh Nauk},
    VOLUME = {38},
      YEAR = {1983},
    NUMBER = {4(232)},
     PAGES = {133--187},
      ISSN = {0042-1316},
   MRCLASS = {58F12 (35B40)},
  MRNUMBER = {710119},
MRREVIEWER = {Yu.\ V.\ Kostarchuk},
}

@article{Foias1989,
author = {Foias, Ciprian and Sell, George R. and Titi, Edriss S.},
title = {Exponential Tracking and Approximation of Inertial Manifolds for Dissipative Nonlinear Equations},
journal = {Journal of Dynamics and Differential Equations},
volume = {1},
number = {2},
pages = {199--244},
year = {1989},
publisher = {Springer},
doi = {10.1007/BF01046904}
}

@article {NPR25,
    AUTHOR = {Nickl, Richard and Pavliotis, Grigorios A. and Ray, Kolyan},
     TITLE = {Bayesian nonparametric inference in {M}c{K}ean-{V}lasov
              models},
   JOURNAL = {Ann. Statist.},
  FJOURNAL = {The Annals of Statistics},
    VOLUME = {53},
      YEAR = {2025},
    NUMBER = {1},
     PAGES = {170--193},
      ISSN = {0090-5364,2168-8966},
   MRCLASS = {62G08 (35Q70 35Q84 62C10 62F15 62G20)},
  MRNUMBER = {4865012},
MRREVIEWER = {Sreenivasan\ Ravi},
       DOI = {10.1214/24-aos2459},
}

@article{CN26,
  author  = {Castre, Aur\'elien and Nickl, Richard},
  title   = {On gradient stability in nonlinear PDE models and inference in interacting particle systems},
  journal = {ArXiv preprint arXiv:2601.10326},
  year    = {2026},
}

@article{NS26,
  author  = {Nickl, Richard and Seizilles, Fanny},
  title   = {Inferring diffusivity from killed diffusion},
  journal = {Annals of Statistics},
  pages = {to appear},
  year    = {2026},
}

@article{Marion1987,
  author    = {Marion, Martine},
  title     = {Attractors for reaction--diffusion equations: existence and estimate of their dimension},
  journal   = {Applicable Analysis},
  year      = {1987},
  volume    = {25},
  number    = {1-2},
  pages     = {101--147},
  doi       = {10.1080/00036818708839678},
  publisher = {Taylor \& Francis}
}

@article {KonNic26,
    AUTHOR = {Konen, D. and Nickl, R.},
     TITLE = {Data assimilation with the {$2D$} {N}avier-{S}tokes equations: {O}ptimal {G}aussian asymptotics for the posterior measure},
   JOURNAL = {Annals of Statistics},
  FJOURNAL = {},
    VOLUME = {},
      YEAR = {2026},
    NUMBER = {},
     PAGES = {to appear}
}

@article {GW25,
    AUTHOR = {Giordano, Matteo and Wang, Sven},
     TITLE = {Statistical algorithms for low-frequency diffusion data: a
              {PDE} approach},
   JOURNAL = {Ann. Statist.},
  FJOURNAL = {The Annals of Statistics},
    VOLUME = {53},
      YEAR = {2025},
    NUMBER = {3},
     PAGES = {1150--1175},
      ISSN = {0090-5364,2168-8966},
   MRCLASS = {62G05 (35R60 62F15 62M15)},
  MRNUMBER = {4925119},
MRREVIEWER = {Sreenivasan\ Ravi},
       DOI = {10.1214/25}
}

@article {MvdMvdV26,
    AUTHOR = {Magra, Adel and van der Meulen, Frank and van der Vaart, Aad},
     TITLE = {Semi-parametric {B}ernstein-von {M}ises theorem in a parabolic {PDE} problem},
   JOURNAL = {arXiv preprint arXiv:2602.00889},
  FJOURNAL = {},
    VOLUME = {},
      YEAR = {2026},
    NUMBER = {},
     PAGES = {},
}

@article {Kon25Minimax,
    AUTHOR = {Konen, D.},
     TITLE = {Inverting the {F}isher information operator in non-linear models},
   JOURNAL = {arXiv preprint arXiv:2601.13254},
  FJOURNAL = {},
    VOLUME = {},
      YEAR = {2026},
    NUMBER = {},
     PAGES = {},
}

@article {CDRS09,
    AUTHOR = {Cotter, S. L. and Dashti, M. and Robinson, J. C. and Stuart,
              A. M.},
     TITLE = {Bayesian inverse problems for functions and applications to
              fluid mechanics},
   JOURNAL = {Inverse Problems},
  FJOURNAL = {Inverse Problems. An International Journal on the Theory and
              Practice of Inverse Problems, Inverse Methods and Computerized
              Inversion of Data},
    VOLUME = {25},
      YEAR = {2009},
    NUMBER = {11},
     PAGES = {115008, 43},
}

@article {S10,
    AUTHOR = {Stuart, A. M.},
     TITLE = {Inverse problems: a {B}ayesian perspective},
   JOURNAL = {Acta Numer.},
  FJOURNAL = {Acta Numerica},
    VOLUME = {19},
      YEAR = {2010},
     PAGES = {451--559},
}

@book{Robinson2001,
	author = {Robinson, J.},
	publisher = {Cambridge Univ. Press},
	title = {Infinite-Dimensional Dynamical Systems: An introduction to dissipative parabolic {PDE}s and the theory of global attractors},
	year = {2001}}

@book {LSZ15,
    AUTHOR = {Law, Kody and Stuart, Andrew and Zygalakis, Konstantinos},
     TITLE = {Data assimilation},
 PUBLISHER = {Springer, Cham},
      YEAR = {2015},
     PAGES = {xviii+242},
}

@book {RC15,
    AUTHOR = {Reich, Sebastian and Cotter, Colin},
     TITLE = {Probabilistic forecasting and {B}ayesian data assimilation},
 PUBLISHER = {Cambridge University Press, New York},
      YEAR = {2015},
     PAGES = {x+297},
}

@book{NicklEMS,
	author = {Nickl, R.},
	publisher = {EMS Press, Berlin},
	series = {EMS Zurich Lectures in Advanced Mathematics},
	title = {Bayesian non-linear statistical inverse problems},
	year = {2023}}

@article{NicTit2024,
	author = {Nickl, R. and Titi, E.},
	journal = {Ann. Statist.},
	number = {4},
	pages = {1825--1844},
	title = {On posterior consistency of data assimilation with {G}aussian process priors: the 2{D} {N}avier-{S}tokes equations},
	volume = {52},
	year = {2024}}

@article{KSV24,
	author = {Koers, G and Szabo, B. and van der Vaart, A.W.},
	journal = {Annals of Statistics},
	title = {Linear methods for non-linear inverse problems},
    pages={to appear},
	year = {2025}
}

@article{Nickl2024,
	author = {Nickl, R.},
	journal = {arXiv preprint arXiv:2407.14781},
	title = {Bernstein-von {M}ises theorems for time evolution equations},
	year = {2024}}

@article {B23,
    AUTHOR = {Bohr, Jan},
     TITLE = {A {B}ernstein--von-{M}ises theorem for the {C}alder\'on
              problem with piecewise constant conductivities},
   JOURNAL = {Inverse Problems},
  FJOURNAL = {Inverse Problems. An International Journal on the Theory and
              Practice of Inverse Problems, Inverse Methods and Computerized
              Inversion of Data},
    VOLUME = {39},
      YEAR = {2023},
    NUMBER = {1},
     PAGES = {Paper No. 015002, 18},
      ISSN = {0266-5611,1361-6420},
   MRCLASS = {35R60 (62F15)},
  MRNUMBER = {4527131},
}

@article {AN19,
    AUTHOR = {Abraham, Kweku and Nickl, Richard},
     TITLE = {On statistical {C}alder\'on problems},
   JOURNAL = {Math. Stat. Learn.},
  FJOURNAL = {Mathematical Statistics and Learning},
    VOLUME = {2},
      YEAR = {2019},
    NUMBER = {2},
     PAGES = {165--216},
      ISSN = {2520-2316,2520-2324},
   MRCLASS = {35R30 (35J25 62F15 62G05)},
  MRNUMBER = {4130599},
}

@book{van1998,
	address = {Cambridge},
	author = {van der Vaart, A. W.},
	publisher = {Cambridge Univ. Press},
	title = {Asymptotic Statistics},
	year = {1998}}

@article {MNP21a,
    AUTHOR = {Monard, Francois and Nickl, Richard and Paternain, Gabriel
              P.},
     TITLE = {Consistent inversion of noisy non-{A}belian {X}-ray
              transforms},
   JOURNAL = {Comm. Pure Appl. Math.},
  FJOURNAL = {Communications on Pure and Applied Mathematics},
    VOLUME = {74},
      YEAR = {2021},
    NUMBER = {5},
     PAGES = {1045--1099},
      ISSN = {0010-3640,1097-0312},
}


\bigskip

\textsc{Department of Pure Mathematics \& Mathematical Statistics}

\textsc{University of Cambridge}, Cambridge, UK

Email: dimitri.konen@ulb.be, nickl@maths.cam.ac.uk







\end{document}